\documentclass[12pt]{article}
\usepackage{amsmath,amsthm,amssymb}
\usepackage{amssymb,latexsym}
\newtheorem{theorem}{Theorem}

\newtheorem{lemma}{Lemma}

\begin{document}

\baselineskip=17pt

\title{\bf Primes of the form {\boldmath$[\textrm{n}^c]$} with  square-free n }

\author{\bf S. I. Dimitrov}

\date{2022}

\maketitle

\begin{abstract}
Let $[\, \cdot\,]$ be the floor function.
In this paper we show that when $1<c<\frac{3849}{3334}$, then there exist infinitely many prime numbers of the form
$[n^c ]$, where $n$ is square-free.
\quad\\
{\bf  2020 Math.\ Subject Classification}:  11L07 $\cdot$ 11N25
\end{abstract}

\section{Notations}
\indent

Let $x$ be a sufficiently large positive number.
The letter $p$  with or without subscript will always denote prime number.
By $\varepsilon$ we denote an arbitrary small positive constant.
As usual $\mu (n)$ and $\Lambda(n)$ denote respectively M\"{o}bius' function and von Mangoldt's function.
The notation $m\sim M$ means that $m$ runs through the interval $(M, 2M]$.
As usual $[t]$ and $\{t\}$ denote the integer part, respectively, the fractional part of $t$.
Moreover $e(y)=e^{2\pi i y}$ and $\psi(t)=\{t\}-1/2$.
We denote by $\tau _k(n)$ the number of solutions of the equation $m_1m_2\ldots m_k$ $=n$ in natural numbers $m_1,\,\ldots,m_k$.
Instead of $m\equiv n\,\pmod {k}$ we write for simplicity $m\equiv n\,(k)$.
We assume that $1<c<\frac{3849}{3334}$ and $\gamma=\frac{1}{c}$.
Denote
\begin{align}
\label{Scx}
&S_c(x)=\sum\limits_{n\leq x\atop{[n^c]=p}}\mu^2(n)\,;\\
\label{Scxwidetilde}
&\widetilde{S_c}(x)=\sum\limits_{n\leq x\atop{[n^c]=p}}\mu^2(n)\mu^2(n+1)\,;\\
\label{sigma}
&\sigma=\prod\limits_{p}\left(1-\frac{2}{p^2}\right)\,;\\
\label{z}
&z=\log^2x\,.
\end{align}

\section{Introduction and statement of the result}
\indent

The problems for the existence of infinitely many prime numbers of a special form are as 
interesting as well as difficult in prime number theory.
One of them is the representation of infinitely many prime numbers by polynomials.
It is conjectured that if $f(x)$ is any irreducible integer polynomial such that $f(1)$, $f(2)$, $\ldots$
tend to infinity and have no common factor greater than 1, then $f(n)$ takes infinitely many prime values. 
This problem has been completely solved for linear polynomials by the Dirichlet theorem on primes in
arithmetic progressions. Now it remains unsolved for polynomials of degree greater
than 1 and it seems to be out of reach of the current state of the mathematics.
For this reason, the mathematical world deals with the accessible problem of primes of the form $[n^c]$.
Let $\mathbb{P}$ denotes the set of all prime numbers.
In 1953 Piatetski-Shapiro \cite{Shapiro} has shown that for any fixed $1<c<\frac{12}{11}$ the set
\begin{equation*}
\mathbb{P}_c=\{p\in\mathbb{P}\;\;|\;\; p= [n^c]\;\; \mbox{ for some } n\in \mathbb{N}\}
\end{equation*}
is infinite.
The prime numbers of the form $p=[n^c]$ are called Piatetski-Shapiro primes.
Denote
\begin{equation*}
\pi_c(x)=\sum\limits_{n\leq x\atop{[n^c]=p}}1\,.
\end{equation*}
Piatetski-Shapiro's result states that
\begin{equation}\label{Shapiroformula}
\pi_c(x)=\frac{x}{c\log x}+\mathcal{O}\left(\frac{x}{\log^2x}\right)
\end{equation}
for
\begin{equation*}
1<c<\frac{12}{11}\,.
\end{equation*}
Subsequently the interval for $c$ was sharpened many times \cite{Baker-Harman}, \cite{Heath2}, \cite{Jia1},
\cite{Jia2}, \cite{Kolesnik1}, \cite{Kolesnik2}, \cite{Kumchev}, \cite{Leitmann}, \cite{Liu-Rivat}, \cite{Rivat}.
To achieve a longer interval for $c$ the authors used the fact that the upper bound for $c$ is closely connected
with the estimate of an exponential sum over primes.
The best results up to now belongs to Rivat and Sargos \cite{Rivat-Sargos} with \eqref{Shapiroformula}
for
\begin{equation*}
1<c<\frac{2817 }{2426}
\end{equation*}
and to Rivat and Wu \cite{Rivat-Wu} with
\begin{equation*}
\pi_c(x)\gg \frac{x}{\log x}
\end{equation*}
for
\begin{equation*}
1<c<\frac{243}{205}\,.
\end{equation*}
As researchers in additive prime number theory have asked whether different additive questions about the primes can be
resolved in prime numbers from special sets, Piatetski-Shapiro primes have become a favorite "test case" for some results.
Over the last three decades, number theorists have solved various equations and inequalities with Piatetski-Shapiro primes.
On the other hand researchers in multiplicative number theory have studied arithmetic properties of primes of the form $p=[n^c]$.
In 2014 Baker, Banks, Guo and Yeager \cite{Baker-Banks} considered for the first time
Piatetski-Shapiro primes $p=[n^c]$ under imposed conditions on the numbers $n$.
They showed that for any fixed $1<c<\frac{77}{76}$ there are infinitely many
primes of the form $p=[n^c]$, where $n$ is a natural number with at most eight prime factors.
In this connection, in 2016, the article of Banks,  Guo and Shparlinski \cite{Banks} appeared,
which contains a study of the existence of infinitely many prime numbers of the form $[p^c]$.

Inspired by Baker, Banks, Guo, Yeager and Shparlinski we investigate the existence of infinitely many
Piatetski-Shapiro primes $p=[n^c]$, such that $n$ or $n^2+n$ runs through the set of square-free numbers.
We show that for any fixed $1<c<\frac{3849}{3334}$ the sets
\begin{equation*}
\mathbb{T}_c=\{p\in\mathbb{P}\;\;|\;\; p= [n^c]\,, \;\; \mu^2(n)=1 \}
\end{equation*}
\begin{equation*}
\mathbb{\widetilde{T}}_c=\{p\in\mathbb{P}\;\;|\;\; p= [n^c]\,, \;\; \mu^2(n^2+n)=1 \}
\end{equation*}
are infinite.
More precisely we establish the following theorems.
\begin{theorem}\label{Theorem} Let $1<c<\frac{3849}{3334}$. Then for the sum \eqref{Scx} the asymptotic formula
\begin{equation}\label{asymptotic formula1}
S_c(x)=\frac{6}{c\pi^2}\frac{x}{\log x}+\mathcal{O}\left(\frac{x}{\log^2x}\right)
\end{equation}
holds.
\end{theorem}

\begin{theorem} Let $1<c<\frac{3849}{3334}$. Then for the sum \eqref{Scxwidetilde} the asymptotic formula
\begin{equation*}
\widetilde{S_c}(x)=\frac{\sigma x}{c\log x}+\mathcal{O}\left(\frac{x}{\log^2x}\right)
\end{equation*}
holds. Here $\sigma$ is defined by \eqref{sigma}.
\end{theorem}
The proof of Theorem 2 is not essentially different from the proof of Theorem 1.
For this reason, we will only give the proof of Theorem 1.

\newpage

\section{Preliminary lemmas}
\indent

\begin{lemma}\label{Exponentpairs}
Let $|f^{(m)}(u)|\asymp YX^{1-m}$  for $1\leq X<u<X_0\leq2X$ and $m\geq1$.
Then
\begin{equation*}
\bigg|\sum_{X<n\le X_0}e(f(n))\bigg|
\ll Y^\varkappa X^\lambda +Y^{-1}\,,
\end{equation*}
where $(\varkappa, \lambda)$ is any exponent pair.
\end{lemma}
\begin{proof}
See (\cite{Graham-Kolesnik}, Ch. 3).
\end{proof}

\begin{lemma}\label{Robert-Sargosest} Let $H$, $N$, $M$ be positive integers and  $F$ is a real number greater than one.
Let $\alpha$, $\beta$ and $\gamma$ be real numbers such that $\alpha(\alpha-1)\beta\gamma\neq0$.
Set
\begin{equation*}
\Sigma_1=\sum\limits_{h=H+1}^{2H} \sum\limits_{n=N+1}^{2N}
\Bigg|\sum\limits_{m\sim M}e\Bigg(F\frac{m^\alpha h^\beta n^\gamma}{M^\alpha H^\beta N^\gamma}\Bigg)\Bigg|\,.
\end{equation*}
Then
\begin{equation*}
\Sigma_1\ll (HNM)^{1+\varepsilon} \left\{\left(\frac{F}{HNM^2}\right)^{\frac{1}{4}}+\frac{1}{M^{\frac{1}{2}}}+\frac{1}{F} \right\}\,.
\end{equation*}
\end{lemma}
\begin{proof}
See (\cite{Robert-Sargos}, Theorem 3).
\end{proof}

\begin{lemma}\label{Bakerest}  Let $\alpha$, $\alpha_1$ and $\alpha_2$ be real numbers such that
$\alpha<1,\,\alpha\alpha_1\alpha_2\neq0$.
Let $M\geq1$, $M_1\geq1$, $M_2\geq1$. Let $a(m)$ and $b(m_1, m_2)$ be complex numbers with $|a(m)|\leq1$ and $|b(m_1, m_2)|\leq1$.
Set
\begin{equation*}
\Sigma_2=\sum\limits_{m\sim M}\sum\limits_{m_1\sim M_1}\sum\limits_{m_2\sim M_2}a(m)b(m_1, m_2)
e\Bigg(F\frac{m^\alpha m_1^{\alpha_1}m_2^{\alpha_2}}{M^\alpha M_1^{\alpha_1} M_2^{\alpha_2}}\Bigg)\,,
\end{equation*}
where
\begin{equation*}
F\geq  M_1M_2\,.
\end{equation*}
Then
\begin{equation*}
\Sigma_2\ll \big(MM_1M_2\log2M_1M_2\big)\left\{ \frac{1}{(M_1M_2)^{\frac{1}{2}}}
+\left(\frac{F}{M_1M_2}\right)^\frac{\varkappa}{2(1+\varkappa)} \left(\frac{1}{M}\right)^\frac{1+\varkappa-\lambda}{2(1+\varkappa)}\right\}\,,
\end{equation*}
where $(\varkappa, \lambda)$ is any exponent pair.
\end{lemma}
\begin{proof}
See (\cite{Baker}, Theorem 2).
\end{proof}

\begin{lemma}\label{Heath-Brown} Let $G(n)$ be a complex valued function.
Assume further that
\begin{align*}
&P>2\,,\quad P_1\le 2P\,,\quad  2\le U<V\le Z\le P\,,\\
&U^2\le Z\,,\quad 128UZ^2\le P_1\,,\quad 2^{18}P_1\le V^3\,.
\end{align*}
Then the sum
\begin{equation*}
\sum\limits_{P<n\le P_1}\Lambda(n)G(n)
\end{equation*}
can be decomposed into $O\Big(\log^6P\Big)$ sums, each of which is either of Type I
\begin{equation*}
\mathop{\sum\limits_{M<m\le M_1}a(m)\sum\limits_{L<l\le L_1}}_{P<ml\le P_1}G(ml)
\end{equation*}
and
\begin{equation*}
\mathop{\sum\limits_{M<m\le M_1}a(m)\sum\limits_{L<l\le L_1}}_{P<ml\le P_1}G(ml)\log l\,,
\end{equation*}
where
\begin{equation*}
L\ge Z\,,\quad M_1\le 2M\,,\quad L_1\le 2L\,,\quad a(m)\ll \tau _5(m)\log P
\end{equation*}
or of Type II
\begin{equation*}
\mathop{\sum\limits_{M<m\le M_1}a(m)\sum\limits_{L<l\le L_1}}_{P<ml\le P_1}b(l)G(ml)
\end{equation*}
where
\begin{equation*}
U\le L\le V\,,\quad M_1\le 2M\,,\quad L_1\le 2L\,,\quad
a(m)\ll \tau _5(m)\log P\,,\quad b(l)\ll \tau _5(l)\log P\,.
\end{equation*}
\end{lemma}
\begin{proof}
See (\cite{Heath1}).
\end{proof}

\begin{lemma}\label{Vaaler}
For every $H\geq1$, we have
\begin{equation*}
\psi(t)=\sum\limits_{1\leq|h|\leq H}a(h)e(ht)+\mathcal{O}\Bigg(\sum\limits_{|h|\leq H}b(h)e(ht)\Bigg)\,,
\end{equation*}
where
\begin{equation}\label{bh}
a(h)\ll\frac{1}{|h|}\,,\quad b(h)\ll\frac{1}{H}\,.
\end{equation}
\end{lemma}
\begin{proof}
See \cite{Vaaler}.
\end{proof}

\section{Beginning of the proof}
\indent

Our first maneuvers are straightforward.
Using \eqref{Scx}, \eqref{z} and the well-known identity
\begin{equation*}
\mu^2(n)=\sum_{d^2|n}\mu(d)
\end{equation*}
we write
\begin{align}\label{Scest}
S_c(x)&=\sum\limits_{n\leq x\atop{[n^c]=p}}\sum_{d^2|n}\mu(d)
=\sum_{d\leq \sqrt{x}}\mu(d)\sum\limits_{n\leq x\atop{[n^c]=p\atop{ n\equiv 0\,(d^2)}}}1\nonumber\\
&=S^{(1)}_c(x)+S^{(2)}_c(x)\,,
\end{align}
where
\begin{align}
\label{Sc1}
&S^{(1)}_c(x)=\sum_{d\leq z}\mu(d)\sum\limits_{n\leq x\atop{[n^c]=p\atop{ n\equiv 0\,(d^2)}}}1\,,\\
\label{Sc2}
&S^{(2)}_c(x)=\sum_{z<d\leq \sqrt{x}}\mu(d)\sum\limits_{n\leq x\atop{[n^c]=p\atop{ n\equiv 0\,(d^2)}}}1\,.
\end{align}
We shall estimate $S^{(1)}_c(x)$ and $S^{(2)}_c(x)$, respectively, in the sections \ref{SectionSc1} and \ref{SectionSc2}.
In section \ref{Sectionfinal} we shall finalize the proof of Theorem \ref{Theorem}.

\section{Estimation of $\mathbf{S^{(1)}_c(x)}$}\label{SectionSc1}
\indent

From \eqref{Sc1} we obtain
\begin{align}\label{Sc1est1}
S^{(1)}_c(x)&=\sum_{d\leq z}\mu(d)\sum\limits_{p\leq x^c}\sum\limits_{n\leq x\atop{[n^c]=p\atop{n\equiv 0\,(d^2)}}}1
=\sum_{d\leq z}\mu(d)\sum\limits_{p\leq x^c}\sum\limits_{n\leq x\atop{p^\gamma\leq n<(p+1)^\gamma\atop{n\equiv 0\,(d^2)}}}1\nonumber\\
&=\sum_{d\leq z}\mu(d)\sum\limits_{p\leq x^c}\sum\limits_{p^\gamma\leq n<(p+1)^\gamma\atop{n\equiv 0\,(d^2)}}1+\mathcal{O}(z)\nonumber\\
&=\sum_{d\leq z}\mu(d)\sum\limits_{p\leq x^c}\sum\limits_{p^\gamma d^{-2}\leq k<(p+1)^\gamma d^{-2}}1+\mathcal{O}(z)\nonumber\\
&=\sum_{d\leq z}\mu(d)\sum\limits_{p\leq x^c}\Big(\big[-p^\gamma d^{-2}\big]-\big[-(p+1)^\gamma d^{-2}\big]\Big)+\mathcal{O}(z)\nonumber\\
&=S^{(3)}_c(x)+S^{(4)}_c(x)+\mathcal{O}(z)\,,
\end{align}
where
\begin{align}
\label{Sc3}
&S^{(3)}_c(x)=\sum_{d\leq z}\frac{\mu(d)}{d^2}\sum\limits_{p\leq x^c}\big((p+1)^\gamma-p^\gamma\big)\,,\\
\label{Sc4}
&S^{(4)}_c(x)=\sum_{d\leq z}\mu(d)\sum\limits_{p\leq x^c}\big(\psi(-(p+1)^\gamma d^{-2})-\psi(-p^\gamma d^{-2})\big)\,.
\end{align}

\subsection{Asymptotic formula for $\mathbf{S^{(3)}_c(x)}$}
\indent

Bearing in mind \eqref{z}, \eqref{Sc3} and the well-known asymptotic formulas
\begin{align*}
&(p+1)^\gamma-p^\gamma=\gamma p^{\gamma-1}+\mathcal{O}\left(p^{\gamma-2}\right)\,,\\
&\sum\limits_{p\leq x^c}p^{\gamma-1}=\frac{x}{\log x}+\mathcal{O}\left(\frac{x}{\log^2x}\right)\,,\\
&\sum_{d\leq z}\frac{\mu(d)}{d^2}=\frac{6}{\pi^2}+\mathcal{O}\left(\frac{1}{z}\right)
\end{align*}
we derive
\begin{equation}\label{Sc3est}
S^{(3)}_c(x)=\frac{6}{c\pi^2}\frac{x}{\log x}+\mathcal{O}\left(\frac{x}{\log^2x}\right)\,.
\end{equation}

\subsection{Upper bound for $\mathbf{S^{(4)}_c(x)}$}
\indent

By \eqref{Sc4} and Abel's summation formula it follows
\begin{equation}\label{Sc4est1}
S^{(4)}_c(x)\ll zx^{\frac{c}{2}}\log^2x+\sum_{d\leq z}\max_{2\leq t \leq x^c}|\Sigma(t)|\,,
\end{equation}
where
\begin{equation}\label{Sigma}
\Sigma(t)=\sum\limits_{n\leq t}\Lambda(n)\big(\psi(-(n+1)^\gamma d^{-2})-\psi(-n^\gamma d^{-2})\big)\,.
\end{equation}
Splitting the range of $n$ into dyadic subintervals from \eqref{Sc4est1} and \eqref{Sigma} we get
\begin{equation}\label{Sc4est2}
S^{(4)}_c(x)\ll zx^{\frac{c}{2}}\log^2x+(\log x)\sum_{d\leq z}\max_{N\leq x^c}\big|S^{(5)}_c(N)\big|\,,
\end{equation}
where
\begin{equation}\label{Sc5}
S^{(5)}_c(N)=\sum\limits_{n\sim N}\Lambda(n)\big(\psi(-(n+1)^\gamma d^{-2})-\psi(-n^\gamma d^{-2})\big)\,.
\end{equation}
Using the trivial estimate for $S^{(5)}_c(N)$ from \eqref{Sc4est2} and \eqref{Sc5} we have
\begin{equation}\label{Sc4est3}
S^{(4)}_c(x)\ll zx^{\frac{c}{2}}\log^2x+\frac{x}{\log^2x}
+(\log x)\sum_{d\leq z}\max_{\frac{x}{z\log^4x}\leq N\leq x^c}\big|S^{(5)}_c(N)\big|\,.
\end{equation}
Henceforth we will use that
\begin{equation}\label{Nxc}
\frac{x}{z\log^4x}\leq N \leq x^c\,.
\end{equation}
Now \eqref{Sc5} and Lemma \ref{Vaaler} imply
\begin{equation}\label{Sc5est}
S^{(5)}_c(N)=S^{(6)}_c(N)+S^{(7)}_c(N)+S^{(8)}_c(N)\,,
\end{equation}
where
\begin{align}
\label{Sc6}
&S^{(6)}_c(N)=\sum\limits_{n\sim N}\Lambda(n)\sum\limits_{1\leq|h|\leq H}a(h)
\big(e(-h(n+1)^\gamma d^{-2})-e(-hn^\gamma d^{-2})\big)\,,\\
\label{Sc7}
&S^{(7)}_c(N)\ll\sum\limits_{n\sim N}\Lambda(n)\sum\limits_{|h|\leq H}b(h)e(-hn^\gamma d^{-2})\,,\\
\label{Sc8}
&S^{(8)}_c(N)\ll\sum\limits_{n\sim N}\Lambda(n)\sum\limits_{|h|\leq H}b(h)e(-h(n+1)^\gamma d^{-2})\,.
\end{align}
Henceforth we will use that
\begin{equation}\label{dleqz}
d\leq z\,.
\end{equation}
Further we choose
\begin{equation}\label{HxNd}
H=x^{\varepsilon-1}Nd^2\,.
\end{equation}
It is easy to see that \eqref{z}, \eqref{Nxc} and \eqref{HxNd} lead to $H\geq 1$.
From \eqref{z}, \eqref{bh}, \eqref{Nxc}, \eqref{Sc7}, \eqref{dleqz}, \eqref{HxNd}
and Lemma \ref{Exponentpairs} with exponent pair $\left(\frac{1}{2},\frac{1}{2}\right)$ we deduce
\begin{align}\label{Sc7est}
S^{(7)}_c(N)&\ll(\log N)\Bigg(b_0N+\sum\limits_{1\leq|h|\leq H}b(h)\sum\limits_{n\sim N}e(-hn^\gamma d^{-2})\Bigg)\nonumber\\
&\ll(\log N)\Bigg(NH^{-1}+\sum\limits_{1\leq|h|\leq H}b(h)\Big(|h|^\frac{1}{2}N^\frac{\gamma}{2}d^{-1}+|h|^{-1}N^{1-\gamma} d^2\Big)\Bigg)\nonumber\\
&\ll(\log N)\Bigg(NH^{-1}+H^{-1}\sum\limits_{1\leq|h|\leq H}\Big(|h|^\frac{1}{2}N^\frac{\gamma}{2}d^{-1}+|h|^{-1}N^{1-\gamma} d^2\Big)\Bigg)\nonumber\\
&\ll(\log N)\Big(NH^{-1}+H^\frac{1}{2}N^\frac{\gamma}{2}d^{-1}+N^{1-\gamma}d^2\Big)\nonumber\\
&\ll x^{1-\varepsilon}d^{-2}\log x\,.
\end{align}
In the same way for the sum defined by \eqref{Sc8} we obtain
\begin{equation}\label{Sc8est}
S^{(8)}_c(N)\ll x^{1-\varepsilon}d^{-2}\log x\,.
\end{equation}
It remains to estimate the sum $S^{(6)}_c(N)$. By \eqref{bh} and \eqref{Sc6} we have
\begin{equation}\label{Sc6est1}
S^{(6)}_c(N)\ll\sum\limits_{1\leq|h|\leq H}\frac{1}{h}
\bigg|\sum\limits_{n\sim N}\Lambda(n)\Phi_h(n)e(-hn^\gamma d^{-2})\bigg|\,,
\end{equation}
where
\begin{equation*}
\Phi_h(t)=e\big(ht^\gamma d^{-2}-h(t+1)^\gamma d^{-2}\big)-1\,.
\end{equation*}
Bearing in mind the estimates
\begin{equation*}
\Phi_h(t)\ll|h|N^{\gamma-1}d^{-2}\,,  \quad \Phi'_h(t)\ll|h|N^{\gamma-2}d^{-2}
\end{equation*}
for  $t\in[N,2N]$ and using Abel's summation formula from \eqref{Sc6est1} we derive
\begin{align}\label{Sc6est2}
S^{(6)}_c(N)&\ll\sum\limits_{1\leq h\leq H}\frac{1}{h}\bigg|\Phi_h(2N)\sum\limits_{n\sim N}\Lambda(n)e(-hn^\gamma d^{-2})\bigg|\nonumber\\
&+\sum\limits_{1\leq h\leq H}\frac{1}{h}\int\limits_{N}^{2N}\bigg|\Phi'_h(t)\sum\limits_{N<n\leq t}\Lambda(n)e(-hn^\gamma d^{-2})\bigg|\,dt\nonumber\\
&\ll N^{\gamma-1}d^{-2}\big|S^{(9)}_c(N_1)\big|\,,
\end{align}
where
\begin{equation}\label{Sc9}
S^{(9)}_c(N_1)=\sum\limits_{1\leq h\leq H}\bigg|\sum\limits_{N<n\leq N_1}\Lambda(n)e(hn^\gamma d^{-2})\bigg|
\end{equation}
for some number $N_1\in (N, 2N]$.\\
Put
\begin{equation}\label{F}
F=H_1L^\gamma M^\gamma d^{-2}\,.
\end{equation}
\begin{lemma}\label{SIest} Assume that
\begin{equation}\label{Conditions1}
H_1\leq H \,,\quad H_2\sim H_1 \,,\quad   |a(m)|\leq 1 \,,\quad LM\asymp N\,,\quad
L\gg N^{\frac{1}{2}-\gamma}H_1^{-\frac{1}{2}}x^{1-3\varepsilon} \,.
\end{equation}
Set
\begin{equation}\label{SI}
S_I=\sum\limits_{H_1\leq h\leq H_2}\bigg|\sum\limits_{m\sim M}a(m)\sum\limits_{l\sim L}e\big(hm^\gamma l^\gamma d^{-2}\big)\bigg|\,.
\end{equation}
Then
\begin{equation*}
S_I\ll x^{1-2\varepsilon}N^{1-\gamma}\,.
\end{equation*}
\end{lemma}
\begin{proof}
By \eqref{z}, \eqref{Nxc}, \eqref{dleqz}, \eqref{F} and \eqref{Conditions1} it follows $F\geq 1$.
Now \eqref{z}, \eqref{Nxc}, \eqref{dleqz}, \eqref{HxNd}, \eqref{F}, \eqref{Conditions1}, \eqref{SI} and Lemma \ref{Robert-Sargosest} yield
\begin{align*}
S_I&\ll(H_1ML)^{1+\varepsilon} \left\{\left(\frac{F}{H_1ML^2}\right)^{\frac{1}{4}}+\frac{1}{L^{\frac{1}{2}}}+\frac{1}{F} \right\}\\
&\ll x^\varepsilon\Big(HN^\frac{\gamma+3}{4}L^{-\frac{1}{4}}d^{-\frac{1}{2}}+HNL^{-\frac{1}{2}}+d^2N^{1-\gamma}\Big)  \\
&\ll x^{1-2\varepsilon}N^{1-\gamma}\,.
\end{align*}
This proves the lemma.
\end{proof}

\begin{lemma}\label{SIIest} Assume that
\begin{equation}
\begin{split}\label{Conditions2}
&H_1\leq H \,,\quad H_2\sim H_1 \,,\quad   |a(m)|\leq 1\,, \quad |b(l)|\leq1\,, \\
&LM\asymp N\,,\quad N^{2\gamma}H_1x^{6\varepsilon-2} \ll L\ll N^{\frac{1}{3}}\,.
\end{split}
\end{equation}
Set
\begin{equation}\label{SII}
S_{II}=\sum\limits_{H_1\leq h\leq H_2}\bigg|\sum\limits_{m\sim M}a(m)\sum\limits_{l\sim L}b(l)e\big(hm^\gamma l^\gamma d^{-2}\big)\bigg|\,.
\end{equation}
Then
\begin{equation*}
S_{II}\ll x^{1-2\varepsilon}N^{1-\gamma}\,.
\end{equation*}
\end{lemma}
\begin{proof}
By \eqref{z}, \eqref{Nxc}, \eqref{dleqz}, \eqref{F} and \eqref{Conditions2} it follows $F\geq LH_1$.
Now \eqref{z}, \eqref{Nxc}, \eqref{dleqz}, \eqref{HxNd}, \eqref{F}, \eqref{Conditions2}, \eqref{SII},
Lemma \ref{Exponentpairs} with exponent pair
\begin{equation*}
BA^5BA^2BA^2B(0, 1)=\left(\frac{480}{1043}\,,\,\frac{528}{1043}\right)
\end{equation*}
and Lemma \ref{Robert-Sargosest} give us
\begin{align*}
S_{II}&\ll \big(H_1ML\log2LH_1\big)\left\{ \frac{1}{(LH_1)^{\frac{1}{2}}}
+\left(\frac{F}{LH_1}\right)^\frac{240}{1523} \left(\frac{1}{M}\right)^\frac{995}{3046}\right\}\\
&\ll x^\varepsilon\left\{ H^\frac{1}{2}NL^{-\frac{1}{2}}+HN\left(\frac{N^\gamma}{Ld^2}\right)^\frac{240}{1523} \left(\frac{1}{M}\right)^\frac{995}{3046}\right\}  \\
&\ll x^{1-2\varepsilon}N^{1-\gamma}\,.
\end{align*}
This proves the lemma.
\end{proof}

\begin{lemma}\label{Sc9est}  For the sum denoted by \eqref{Sc9} we have
\begin{equation*}
S^{(9)}_c(N_1)\ll x^{1-\varepsilon}N^{1-\gamma}\,.
\end{equation*}
\end{lemma}
\begin{proof}
Splitting the range of $h$ into dyadic subintervals from \eqref{Sc9} we get
\begin{equation}\label{Sc9est1}
S^{(9)}_c(N_1)\ll |S^{(10)}_c(N_1)|\log x\,,
\end{equation}
where
\begin{equation}\label{Sc10}
S^{(10)}_c( N_1)=\sum\limits_{h\sim H_1}\bigg|\sum\limits_{N<n\leq N_1}\Lambda(n)e(hn^\gamma d^{-2})\bigg|
\end{equation}
and $H_1\leq H/2$.
Put
\begin{equation}\label{Conditions3}
U=N^{2\gamma}H_1x^{6\varepsilon-2}\,,\quad V=N^{\frac{1}{3}}\,,\quad
Z=\Big[N^{\frac{1}{2}-\gamma}H_1^{-\frac{1}{2}}x^{1-3\varepsilon}\Big]+\frac{1}{2}\,.
\end{equation}
Using \eqref{Sc10}, \eqref{Conditions3}, Lemma \ref{Heath-Brown}, Lemma \ref{SIest} and  Lemma \ref{SIIest} we deduce
\begin{equation}\label{Sc10est}
S^{(10)}_c(N_1)\ll x^{1-\frac{3\varepsilon}{2}}N^{1-\gamma}\,.
\end{equation}
Now \eqref{Sc9est1} and \eqref{Sc10est} imply the proof of the lemma.
\end{proof}
Taking into account \eqref{z}, \eqref{Sc4est3}, \eqref{Sc5est}, \eqref{Sc7est}, \eqref{Sc8est}, \eqref{Sc6est2}
and Lemma \ref{Sc9est} we find
\begin{equation}\label{Sc4est4}
S^{(4)}_c(x)\ll\frac{x}{\log^2x}\,.
\end{equation}

\subsection{Asymptotic formula for $\mathbf{S^{(1)}_c(x)}$}
\indent

From \eqref{z}, \eqref{Sc1est1}, \eqref{Sc3est} and \eqref{Sc4est4} it follows
\begin{align}\label{Sc1est2}
&S^{(1)}_c(x)=\frac{6}{c\pi^2}\frac{x}{\log x}+\mathcal{O}\left(\frac{x}{\log^2x}\right)\,.
\end{align}

\section{Upper bound for $\mathbf{S^{(2)}_c(x)}$}\label{SectionSc2}
\indent

By \eqref{Sc2} we obtain
\begin{align}\label{Sc2est}
&S^{(2)}_c(x)\ll\sum_{z<d\leq \sqrt{x}}\sum\limits_{n\leq x\atop{ n\equiv 0\,(d^2)}}1
\ll x\sum_{z<d\leq \sqrt{x}}\frac{1}{d^2}\ll xz^{-1}\,.
\end{align}

\section{The end of the proof}\label{Sectionfinal}

Summarizing \eqref{z}, \eqref{Scest}, \eqref{Sc1est2} and \eqref{Sc2est}
we establish asymptotic formula \eqref{asymptotic formula1}.

This completes the proof of  Theorem \ref{Theorem}.

\vskip20pt
\footnotesize
\begin{flushleft}
S. I. Dimitrov\\
Faculty of Applied Mathematics and Informatics\\
Technical University of Sofia \\
Blvd. St.Kliment Ohridski 8, \\
Sofia 1756, Bulgaria\\
e-mail: sdimitrov@tu-sofia.bg\\
\end{flushleft}

\end{document}